\documentclass[12pt]{article}
\usepackage{amsthm, amsmath, amssymb}
\usepackage{enumerate}

\usepackage{color}
\definecolor{darkblue}{rgb}{0,0,0.6}
\definecolor{darkgreen}{rgb}{0,0.35,0}
\usepackage[letterpaper,breaklinks,pdftex,colorlinks,citecolor=darkgreen,linkcolor=darkblue]{hyperref}

\newtheorem{theorem}{Theorem}[section]
\newtheorem{conjecture}[theorem]{Conjecture}

\newtheoremstyle{parentheses}{\topsep}{\topsep}{\itshape}{}{}{}{ }{\thmnumber{(#2)}}
\theoremstyle{parentheses}
\newtheorem{claim}{}
\makeatletter\@addtoreset{claim}{theorem}\makeatother

\renewcommand{\d}{\backslash}

\DeclareMathOperator{\si}{si}
\DeclareMathOperator{\cl}{cl}
\DeclareMathOperator{\GF}{GF}
\DeclareMathOperator{\PG}{PG}
\DeclareMathOperator{\AG}{AG}
\DeclareMathOperator{\F}{\mathbb{F}}

\title{Representability of matroids with a large projective geometry minor\thanks{This research was partially supported by a grant from the Office of Naval Research [N00014-10-1-0851].}}
\author{Jim Geelen and Rohan Kapadia\footnote{Current address: Concordia University, Montr\'eal, Qu\'ebec, Canada} \\ Department of Combinatorics and Optimization \\ University of Waterloo, Waterloo, Ontario, Canada}

\date{December 20, 2012; revised December 8, 2014}

\begin{document}

\maketitle

\begin{abstract}
We prove that, for each prime power $q$, there is an integer $n$ such that if $M$ is a $3$-connected, representable matroid with a $\PG(n-1,q)$-minor and no
$U_{2,q^2+1}$-minor, then $M$ is representable over $\GF(q)$.
We also show that for $\ell \geq 2$, if $M$ is a $3$-connected, representable matroid of sufficiently high rank with no $U_{2,\ell+2}$-minor and $|E(M)| \geq 2\ell^{r(M)/2}$, then $M$ is representable over a field of order at most $\ell$.
\end{abstract}

\section{Introduction}

We recall that $\PG(n-1,q)$ is the rank-$n$ projective geometry over $\GF(q)$, the finite field of order $q$. We prove the following theorem.

\begin{theorem} \label{thm:mainpgtheorem}
For each prime power $q$ there is an integer $n$ such that if $M$ is a $3$-connected, representable matroid with a $\PG(n-1,q)$-minor, then either
\begin{itemize}
 \item $M$ has a $U_{2, q^2+1}$-minor, or
 \item $M$ is $\GF(q)$-representable.
\end{itemize}
\end{theorem}

Note that $U_{2, q^2+1}$ is the longest line representable over $\GF(q^2)$.
In the $q = 2$ case, a precise version of \autoref{thm:mainpgtheorem} is known.

\begin{theorem}[Semple, Whittle, \cite{SempleWhittle}] \label{thm:binarycase}
If $M$ is a $3$-connected, representable matroid with a $\PG(2,2)$-minor and no $U_{2,5}$-minor, then $M$ is binary.
\end{theorem}

Semple and Whittle proved that a $3$-connected, representable matroid with no $U_{2,5}$- or $U_{3,5}$-minor is binary or ternary. \autoref{thm:binarycase} follows from this result along with the fact that $\PG(2,2)$ is not ternary and a result of Oxley \cite{Oxley:acharacterization} stating that a $3$-connected matroid of corank at least three with no $U_{2,5}$-minor
has no $U_{3,5}$-minor.

In the next section we exhibit counterexamples to the stronger version of \autoref{thm:mainpgtheorem} where the assumption of representability is dropped. These matroids are vertically $4$-connected and have no $U_{2,q+3}$-minor. However, they have $U_{2,q+2}$-minors, and we conjecture the following.

\begin{conjecture}
If $q$ is a prime power and $M$ is a vertically $4$-connected matroid with a $\PG(2,q)$-minor and no $U_{2, q+2}$-minor, then $M$ is $\GF(q)$-representable. 
\end{conjecture}

For a matroid $M$, we denote the simplification of $M$ by $\si(M)$ and we let $\varepsilon(M) = |E(\si(M))|$.
For a class of matroids $\mathcal{M}$ and positive integer $k$, we define $g_{\mathcal{M}}(k) = \max\{\varepsilon(M) : M \in \mathcal{M}, r(M) = k\}$
or say $g_{\mathcal{M}}(k) = \infty$ when this maximum does not exist. The function $g_\mathcal{M}$ is called the \emph{growth-rate function} of $\mathcal{M}$.
We let $\mathcal{U}_\ell$ denote the class of matroids with no $U_{2,\ell+2}$-minor. A theorem of Geelen and Nelson \cite{GeelenNelson} asserts that, for sufficiently large $k$, $g_{\mathcal{U}_\ell}(k) = (q^k - 1)/(q-1)$ where $q$ is the largest prime power less than or equal to $\ell$, and equality is achieved only by the projective geometry $\PG(k-1,q)$.
Thus, for large $k$, the rank-$k$ matroids in $\mathcal{U}_\ell$ with the maximum number of points are representable over a field of order at most $\ell$.
We prove the following extension of this fact as a corollary of \autoref{thm:mainpgtheorem} and a result of Geelen and Kabell \cite{GeelenKabell}.

\begin{theorem} \label{thm:growthrateapplication0}
For any positive integer $\ell$, there is an integer $k$ so that if $M$ is a $3$-connected, representable matroid of rank at least $k$ with no $U_{2,\ell+2}$-minor and $|E(M)| \geq 2\ell^{r(M)/2}$, then $M$ is representable over a field of order at most $\ell$.
\end{theorem}

\section{Growth rates}

We prove \autoref{thm:mainpgtheorem} in Sections~\ref{sec:representationoverasubfield} and \ref{sec:proofofmainpgtheorem}; in this section we derive \autoref{thm:growthrateapplication0} from it and also provide an example that motivates the assumption of representability in both Theorems~\ref{thm:mainpgtheorem} and \ref{thm:growthrateapplication0}.
We need the following result.

\begin{theorem}[Geelen, Kabell, \cite{GeelenKabell}] \label{thm:densematroidhaspgminor}
For all integers $\ell, q_0 \geq 2$ and $n$, there exists an integer $c$ such that if $M$ is a matroid with no $U_{2,\ell+2}$-minor and $\varepsilon(M) \geq cq_0^{r(M)}$, then $M$ has a $\PG(n-1,q)$-minor for some prime power $q > q_0$.
\end{theorem}

We prove the following stronger version of \autoref{thm:growthrateapplication0}.

\begin{theorem} \label{thm:growthrateapplication}
Let $\ell \geq 2$ and $q_0$ the smallest prime power greater than or equal to $\sqrt{\ell}$.
There is an integer $c$ such that if $M$ is a $3$-connected, representable matroid with no $U_{2,\ell+2}$-minor and $|E(M)| \geq cq_0^{r(M)}$, then $M$ is $\GF(q)$-representable for some prime power $q \leq \ell$.
\end{theorem}

\begin{proof}
Applying \autoref{thm:mainpgtheorem} to every prime power $q \leq \ell$, we can choose an integer $n$ so that, for each such $q$, a $3$-connected, representable matroid with a $\PG(n-1,q)$-minor and no $U_{2,q^2+1}$-minor is representable over $\GF(q)$.
We choose $c$ as in \autoref{thm:densematroidhaspgminor} so that a matroid $M$ with no $U_{2,\ell+2}$-minor and $\varepsilon(M) \geq cq_0^{r(M)}$ has a $\PG(n-1,q)$-minor for some prime power $q > q_0$.
We let $M$ be a $3$-connected, representable matroid with no $U_{2,\ell+2}$-minor and $\varepsilon(M) \geq cq_0^{r(M)}$; then $M$ has a $\PG(n-1,q)$-minor for some prime power $q > q_0$. The fact that $M$ has no $U_{2,\ell+2}$-minor implies that $q \leq \ell$.
Also, $q > \sqrt{\ell}$ so $M$ has no $U_{2,q^2+1}$-minor. Thus by \autoref{thm:mainpgtheorem}, $M$ is $\GF(q)$-representable.
\end{proof}

For any prime power $q$, we exhibit a class of matroids which provide a counterexample to the stronger versions of both Theorems~\ref{thm:mainpgtheorem} and \ref{thm:growthrateapplication} where the assumptions of representability are dropped (for $\ell \geq 4$ in the case of \autoref{thm:growthrateapplication}).
We will use the following theorem of projective geometry, known as Pappus's Theorem (see \cite[Theorem 2.2.2]{BeutelspacherRosenbaum}).

\begin{theorem}
Let $L_1$ and $L_2$ be lines in a plane representable over a field, with distinct points $a,b,c \in L_1 \setminus L_2$ and $d,e,f \in L_2 \setminus L_1$. If $g, h$, and $i$ are points that are respectively the intersections of the lines spanned by $\{e, a\}$ and $\{f, b\}$, $\{d, a\}$ and $\{f, c\}$, and $\{d, b\}$ and $\{e, c\}$, then $g,h$, and $i$ are collinear.
\end{theorem}

For each $n \geq 3$, we construct a rank-$(n+1)$ matroid that is $3$-connected, has a $\PG(n-1,q)$-minor, has no $U_{2,q+3}$-minor, and has more than $q^n$ points, but is not representable.

We recall that the rank-$n$ affine geometry $\AG(n-1, q)$ is obtained from $\PG(n-1, q)$ by deleting a hyperplane, and for any element $e$ of $\AG(n-1, q)$, $\si(\AG(n-1, q) / e) \cong \PG(n-2, q)$. 

For $n \geq 3$, we let $H$ be a hyperplane of $\PG(n,q)$, let $C$ be a circuit of size $n+1$ contained in $H$, let $M_n = \PG(n,q) \d (H \setminus C)$, and let $M_n'$ be the matroid obtained from $M_n$ by relaxing the circuit-hyperplane $C$.
For any element $e$ of $M_n'$, at least one of $M_n' \d e$ and $M_n' / e$ is $\GF(q)$-representable. In particular, if $e \not\in C$, then $\si(M_n' / e) = \si(M_n / e)
\cong \PG(n-1,q)$, and if $e \in C$, then $M_n' \d e = M_n \d e$. However, $M_n'$ has no $U_{2,q+3}$-minor, because any non-$\GF(q)$-representable, rank-$2$ minor $N$
of $M_n'$ is a restriction of $M_n' / X$ for some $X \subseteq C$ with $|X| = n - 1$ and so is a single-element extension of a rank-$2$ minor of $\PG(n,q)$.

We now show that $M_n'$ is not representable.
We choose a set $X \subseteq C$ with $|X| = |C| - 3$ and let $N = \si(M_n/X)$ and $N'= \si(M_n'/X)$. Then
$N'$ is obtained from $N$ by relaxing the circuit-hyperplane $C \setminus X$.
If $q = 2$, then $N \cong \PG(2,2)$ so $N'$ is isomorphic to the non-Fano matroid, and hence $M_n'$ is not representable.
If $q > 2$, we label the elements of $C \setminus X$ as $a, b$, and $c$, and we can choose a triangle $\{d,e,f\}$ of $N$ such that $a,b,c \not\in \cl_N(\{d,e,f\})$.
In addition, we can define $g,h$ and $i$ to be the elements of $N$ that respectively lie in $\cl_N(\{e,a\}) \cap \cl_N(\{f,b\})$, $\cl_N(\{d,a\}) \cap
\cl_N(\{f,c\})$, and $\cl_N(\{d,b\}) \cap \cl_N(\{e,c\})$.
We observe that $r_N(\{g,h,i\}) = 2$ by Pappus's Theorem.
Therefore, in $N'$ there are two triangles $\{d,e,f\}$ and $\{g,h,i\}$ that lie on distinct lines, and $a \in \cl_{N'}(\{d,h\}) \cap \cl_{N'}(\{e,g\})$, $b \in
\cl_{N'}(\{d,i\}) \cap \cl_{N'}(\{f,g\})$, and $c \in \cl_{N'}(\{e,i\}) \cap \cl_{N'}(\{f,h\})$.
If $N'$ is representable over a field, then Pappus's Theorem asserts that $a,b$ and $c$ are collinear. But $r_{N'}(\{a,b,c\}) = 3$, so $N'$ is not representable.

\section{Representation over a subfield} \label{sec:representationoverasubfield}

We say that a representation $A$ of a matroid $M$ is in \emph{standard form with respect to} a basis $B$ if it has the form $A = [I \; A']$ where $I$ is an identity matrix in the columns indexed by $B$. For such a representation, we index the rows by the elements of $B$ so that $A_{bb} = 1$ for all $b \in B$. 
When $X \subseteq B$ and $Y \subseteq E(M)$, we write $A[X, Y]$ for the submatrix of $A$ in the rows of $X$ and the columns of $Y$.
For each basis $B$ of a matroid $M$, every representation of $M$ can be converted to standard form with respect to $B$ by applying row operations and permuting the columns along with their labels.

Let $N$ be a minor of a matroid $M$ such that $N = M / C \d D$ for disjoint sets $C, D \subseteq E(M)$ where $C$ is independent and $D$ is coindependent. We choose a basis $B$ of $N$ and let $B' = B \cup C$, so $B'$ is a basis of $M$.
Let $\F$ be a field and $A'$ an $\F$-representation of $M$ in standard form with respect to the basis $B'$.
Then the matrix $A = A'[B, E(N)]$ is an $\F$-representation of $N$ in standard form with respect to the basis $B$.
We say that $A$ is the representation of $N$ \emph{induced} by $A'$, and that $A'$ is a representation of $M$ that \emph{extends} the representation $A$ of $N$.

We call both row operations and column scaling \emph{projective transformations} and say that two representations of a matroid over a field $\F$ are \emph{projectively equivalent} if one can be obtained from the other by applying projective transformations and permuting columns (along with their labels).

A proof of the next result can be found in \cite[Theorem 3.4]{Nelson}.

\begin{theorem} \label{thm:pghasentriesinsubfield}
If $q$ is a prime power, $n \geq 3$, and $\F$ is an extension field of $\GF(q)$, then each representation of $\PG(n-1,q)$ over $\F$ is projectively equivalent
to a representation with entries in $\GF(q)$.
\end{theorem}

When $\F$ is an extension field of $\GF(q)$, we say that an $\F$-matrix $A$ is a \emph{scaled $\GF(q)$-matrix} if there is a $\GF(q)$-matrix obtained from $A$ by scaling rows and columns by elements of $\F^\times$.
\autoref{thm:pghasentriesinsubfield} is equivalent to the fact that for $n \geq 3$, every representation of $\PG(n-1, q)$ in standard form is a scaled $\GF(q)$-matrix. This follows from two observations: when two projectively equivalent representations of a matroid are in standard form with respect to the same basis, then one can be obtained from the other by scaling rows and columns. Also, for $n \geq 3$, $\PG(n-1,q)$ is only representable over extension fields of $\GF(q)$ (see \cite[p. 660]{Oxley}).

We will use the following theorem of Pendavingh and Van Zwam that reduces the problem of proving that a matroid $M$ with a $\PG(n-1,q)$-minor $N$ is $\GF(q)$-representable to checking minors of $M$ with at most $|E(N)| + 2$ elements.
Suppose that $N$ is a minor of an $\F$-representable matroid $M$ and $\F'$ is a subfield of $\F$. We say that $N$ \emph{confines $M$ to $\F'$} if whenever $N'$ is a minor of $M$ isomorphic to $N$, every $\F$-representation of $M$ that extends an $\F'$-representation of $N'$ is a scaled $\F'$-matrix.
Although Pendavingh and Van Zwam prove a theorem for representability over a generalization of fields called partial fields \cite[Theorem 1.4]{PendavinghVanZwam}, we state here only a specialization of it to fields.

\begin{theorem}[Pendavingh, Van Zwam, \cite{PendavinghVanZwam}] \label{thm:confinement}
If $\F'$ is a subfield of a field $\F$, $M$ and $N$ are $3$-connected matroids, and $N$ is a minor of $M$, then either
\begin{enumerate}[(i)]
 \item $N$ confines $M$ to $\F'$, or \label{out:confinement-confines}
 \item $M$ has a $3$-connected minor $M'$ such that \label{out:confinement-hassmallunconfinedminor}
 $N$ does not confine $M'$ to $\F'$ and $N$ is isomorphic to one of $M' / x$, $M' \d y$, or $M' / x \d y$ for some $x, y \in E(M')$.
\end{enumerate}
\end{theorem}

\section{The proof of Theorem \ref*{thm:mainpgtheorem}} \label{sec:proofofmainpgtheorem}

Before proving \autoref{thm:mainpgtheorem} we state a result from Ramsey theory and then a theorem of Tutte about matroid connectivity.
The first is the following corollary of the Hales-Jewett Theorem \cite{HalesJewett}; it is also a special case of the Affine Ramsey Theorem of Graham, Leeb, and Rothschild \cite{GrahamLeebRothschild}, for which a proof can be found in \cite[p. 42]{GrahamRothschildSpencer:ramseytheory}.

\begin{theorem} \label{thm:halesjewettforaffinegeometries}
For any finite field $\GF(q)$ and integers $r$ and $k$, there is an integer $n = n_{\ref{thm:halesjewettforaffinegeometries}}(q, r, k)$ so that if the elements of $\AG(n-1,q)$ are $r$-coloured, it has a monochromatic restriction isomorphic to $\AG(k-1,q)$.
\end{theorem}

The \emph{connectivity function}, $\lambda_M$, of a matroid $M$ is defined by $\lambda_M(X) = r_M(X) + r_M(E(M) \setminus X) - r(M)$ for each $X \subseteq E(M)$. For disjoint sets $S, T \subseteq E(M)$, we define $\kappa_M(S, T) = \min\{\lambda_M(A) : S \subseteq A \subseteq E(M) \setminus T\}$.
When $M$ is a $3$-connected matroid and $S$ and $T$ are disjoint subsets of $E(M)$, both of size at least two, then $\kappa_M(S, T) \geq 2$.
The \emph{local connectivity} of sets $S$ and $T$ in a matroid $M$ is $\sqcap_M(S, T) = r_M(S) + r_M(T) - r_M(S \cup T)$. 

\begin{theorem}[Tutte's Linking Theorem, \cite{Tutte:mengerstheoremformatroids}] \label{thm:tutteslinkingtheorem}
If $M$ is a matroid and $S, T \subseteq E(M)$ are disjoint, then $\kappa_M(S, T) = \max\{ \sqcap_{M/Z}(S,T) : Z \subseteq E(M) \setminus (S \cup T) \}$.
\end{theorem}

Two sets $S$ and $T$ in a matroid $M$ are called \emph{skew} if $\sqcap_M(S,T) = 0$. If we choose the set $Z$ that attains the maximum in \autoref{thm:tutteslinkingtheorem} to be minimal, then $Z$ and $S$ are skew, and $Z$ and $T$ are skew.
We can now prove \autoref{thm:mainpgtheorem}.

\begin{proof}[Proof of \autoref{thm:mainpgtheorem}]
We set $n$ to be the integer $n_{\ref{thm:halesjewettforaffinegeometries}}(q, q^2, 3)$ given by \autoref{thm:halesjewettforaffinegeometries} such that any $q^2$-colouring of the elements $\AG(n-1, q)$ has a monochromatic restriction isomorphic to $\AG(3, q)$.
We let $M$ be a $3$-connected, representable matroid with a $\PG(n-1, q)$-minor but no $U_{2, q^2+1}$-minor. Then $M$ is representable over an extension field $\F$ of $\GF(q)$.
We start with two short claims; we omit the easy proof of the first.

\begin{claim} \label{clm:mainpgtheorem-getlongline}
 If $P$ is a simple rank-$3$ matroid with an element $e$ such that $P \d e \cong \PG(2,q)$, then $P$ has a $U_{2, q^2+1}$-minor.
\end{claim}

\begin{claim} \label{clm:mainpgtheorem-extension}
 If $P$ is an $\F$-representable matroid with an element $y$ such that $P \d y \cong \PG(n-1,q)$ but $\PG(n-1,q)$ does not confine $P$ to $\GF(q)$, then $P$ has a $U_{2, q^2+1}$-minor.
\end{claim}

There is a $\PG(n-1,q)$-minor $N$ of $P$ and an $\F$-representation $A$ of $P$, in standard form with respect to a basis $B$ of $N$, that extends a $\GF(q)$-representation of $N$ but is not a scaled $\GF(q)$-matrix.
The column of $y$ is not parallel to a vector over $\GF(q)$ so there are two elements $a, b \in B$ such that $A_{ay}^{-1}A_{by} \not\in \GF(q)$. We pick any third element $c \in B$, and let $P' = M / (B \setminus \{a, b, c\})$. Then $y$ is not in a parallel pair of $P'$ and $\si(P') \d y \cong \PG(2, q)$, so (\ref{clm:mainpgtheorem-extension}) follows from (\ref{clm:mainpgtheorem-getlongline}).
\\

We apply \autoref{thm:confinement} to $M$ with $N = \PG(n-1,q)$ and $\F' = \GF(q)$. If outcome (\ref{out:confinement-confines}) of this theorem holds, then it follows from \autoref{thm:pghasentriesinsubfield} that $M$ is $\GF(q)$-representable. So we may assume that outcome (\ref{out:confinement-hassmallunconfinedminor}) of \autoref{thm:confinement} holds, and there is a $3$-connected minor $M'$ of $M$ such that $\PG(n-1,q)$ does not confine $M'$ to $\GF(q)$ and $\PG(n-1,q)$ is isomorphic to either $M' / x$, $M' \d y$, or $M' / x \d y$ for some $x, y \in E(M')$.
By (\ref{clm:mainpgtheorem-extension}) we may assume that $M$ has a $\PG(n-1,q)$-minor $N$ equal to either $M' / x$ or $M' / x \d y$ for some $x, y \in E(M')$.

We let $B$ be a basis of $N$ and $A$ be an $\F$-representation of $M'$ in standard form with respect to the basis $B \cup \{x\}$ of $M'$. 
Since $\PG(n-1,q)$ does not confine $M'$ to $\GF(q)$, we may assume that $A$ is not a scaled $\GF(q)$-matrix but it induces a $\GF(q)$-representation $A[B, E(N)]$ of $N$.
Moreover, when $N \cong M' / x \d y$, applying (\ref{clm:mainpgtheorem-extension}) to $M' / x$ lets us assume that $\PG(n-1, q)$ confines $M' / x$ to $\GF(q)$ and that the induced representation $A[B, E(N) \cup \{y\}]$ of $M' / x$ also has all its entries in $\GF(q)$.

\begin{claim} \label{clm:mainpgtheorem-getpair}
 There are two elements $f, g \in E(M' / x)$ such that $A_{xf}, A_{xg} \neq 0$, $A_{xf}^{-1}A_{xg} \not\in \GF(q)$, and $\{f, g\}$ is independent in $M' / x$.
\end{claim}

Let $f$ and $g$ be any two distinct elements of $E(N)$ with $A_{xf}, A_{xg} \neq 0$. Then $\{f, g\}$ is independent in $M' / x$ because $N$ is simple. Therefore, we may assume that $A_{xf}^{-1} A_{xg} \in \GF(q)$ for every pair $f, g \in E(N)$ with $A_{xf}, A_{xg} \neq 0$.
This implies that we can scale the row and column of $x$ in $A$ to transform $A[B \cup \{x\}, E(N) \cup \{x\}]$ into a $\GF(q)$-matrix. But $A$ is not a scaled $\GF(q)$-matrix, so we may assume that we are in the case where $N = M' / x \d y$, that $A_{xy} \neq 0$, and that for any $f \in E(N)$ with $A_{xf} \neq 0$, we have $A_{xf}^{-1} A_{xy} \not\in \GF(q)$.

Note that $y$ is not a loop in $M' / x$ because $M'$ is $3$-connected.
If there exist two distinct elements $f, f' \in E(N)$ with $A_{xf} \neq 0$ and $A_{xf'} \neq 0$, then the fact that $N$ is simple means that at least one of $\{f, y\}$ and $\{f', y\}$ is independent in $M' / x$, and we are done. 
On the other hand, there is at least one element $f \in E(N)$ with $A_{xf} \neq 0$ for otherwise $A$ would be a scaled $\GF(q)$-matrix.
So we may assume that there is precisely one element $f$ of $E(N)$ with $A_{xf} \neq 0$, and that $\{f, y\}$ is a parallel pair of $M' / x$. Now $\{f, x, y\}$ is both a circuit and a cocircuit of $M'$. Hence $\lambda_{M'}(\{f, x, y\}) = 1$, contradicting the fact that $M'$ is $3$-connected. This proves (\ref{clm:mainpgtheorem-getpair}).
\\

We choose a pair of elements $f, g \in E(M' / x)$ as in (\ref{clm:mainpgtheorem-getpair}), and by scaling we may assume that $A_{xf} = 1$ and $A_{xg} = \omega$ for some $\omega \not\in \GF(q)$.
We choose some hyperplane $H$ of $M' / x$ that contains $\{f, g\}$ and choose an element $z \in E(M'/x) \setminus H$. We let $B'$ be the union of $\{z\}$ with a basis of $H$ in $M' / x$, so $B' \cup \{x\}$ is a basis of $M'$, and we let $A'$ be a representation of $M'$ in standard form with respect to $B' \cup \{x\}$. We can obtain $A'$ from $A$ by row operations without using the row of $x$, so that $A'[B', E(M')]$ has all its entries in $\GF(q)$.
We let $C = E(M' / x) \setminus H$, so $C$ is a cocircuit of $M' / x$ containing $z$. Then the restriction $(M' / x) | C$ is isomorphic to $\AG(n-1, q)$. For each $e \in E(M' / x)$, the entry $A'_{ze}$ is non-zero if and only if $e \in C$, and by scaling columns of $A'$ we may assume that all entries in the row of $z$ are either $0$ or $1$.
The submatrix $A'[\{x, z\}, C]$ represents $(M' / (B' \setminus \{z\})) | C$, which has rank two. If this matrix contains a set of at least $q^2+1$ pairwise non-parallel columns, then $M'$, and hence $M$, has a $U_{2, q^2+1}$-minor. Otherwise, since $A'_{ze} = 1$ for all $e \in C$, there are at most $q^2$ distinct elements of $\F$ that appear in $A'[\{x\}, C]$.
We can therefore $q^2$-colour the elements of $(M' / x) | C$ by assigning to each $e \in C$ the colour $A'_{xe}$. Since $(M' / x) | C \cong \AG(n-1, q)$, with our choice of $n = n_{\ref{thm:halesjewettforaffinegeometries}}(q, q^2, 3)$ \autoref{thm:halesjewettforaffinegeometries} implies that there is a monochromatic restriction of $(M' / x) | C$ isomorphic to $\AG(3, q)$. We denote by $Y$ the ground set of this restriction. The entries $A'_{xe}$ for $e \in Y$ are all equal to some $\beta \in \F$, so $A'[\{x\}, Y]$ is a multiple of $A'[\{z\}, Y]$ (possibly the zero multiple) and $M' | Y$ is also isomorphic to $\AG(3, q)$.
Since $f, g \not\in C$, $A'_{zf} = A'_{zg} = 0$, so the row space of $A'$ contains a vector $u \in \F^{E(M')}$ such that $u_e = -\beta$ for all $e \in Y$ and $u_f = u_g = 0$.

As $N$ is $3$-connected, $\kappa_N(\{f, g\}, Y) = 2$. Also, when $N = M' / x \d y$, $\kappa_{M' / x}(\{f, g\}, Y) = 2$ because $y$ is parallel to an element of $N$ in $M' / x$.
By \autoref{thm:tutteslinkingtheorem}, there is a set $Z \subseteq E(M' / x)$ disjoint from $Y$ and $\{f, g\}$ such that $\sqcap_{(M' / x) / Z}(Y, \{f, g\}) = 2$, and $Z$ and $Y$ are skew.
This means that $\{f, g\}$ is independent in $(M' / x) / Z$ and $f, g \in \cl_{(M' / x) / Z}(Y)$.
Since $Z$ and $Y$ are skew, there exists a basis $B''$ of $M' / x$ that contains $Z$ and a basis of $Y$. We apply row operations to $A'$ to get a representation $A''$ of $M'$ in standard form with respect to the basis $B'' \cup \{x\}$. The row of $x$ is the same in $A''$ and $A'$, and the vector $u$ is also in the row space of $A''$.

Consider the matrix $D$ obtained from $A''$ by adding the vector $u$ to the row of $x$ then restricting to the submatrix in rows $\{x\} \cup (B'' \cap Y)$ and columns $Y \cup \{f, g\}$. Then $D$ represents $M'' = (M' / (B'' \setminus Y)) | (Y \cup \{f, g\})$ and it has the form
\[ D = \bordermatrix{& Y & f & g \cr & 0 & 1 & \omega \cr & D_1 & \alpha & \alpha'}, \]
where $D_1$ is a $\GF(q)$-representation of $\AG(3, q)$ and $\alpha$ and $\alpha'$ are columns with all entries in $\GF(q)$.
Since $\{f, g\}$ is independent and contained in the closure of $Y$ in $(M' / x) / Z$, the vectors $\alpha$ and $\alpha'$ are both non-zero and are not parallel to each other.
The minor $M'' / f$ has the following representation
\[ \bordermatrix{& Y & g \cr & D_1 & \alpha' - \omega\alpha}. \]
Since $\omega \not\in \GF(q)$ and $\alpha$ and $\alpha'$ are both non-zero and are not parallel, the column $\alpha' - \omega\alpha$ is not parallel to a vector over $\GF(q)$.
We have $M'' / f \d g \cong \AG(3, q)$. 
Suppose there are two distinct lines $L_1$ and $L_2$ of $M'' / f \d g$ such that $g \in \cl_{M''/f}(L_1) \cap \cl_{M''/f}(L_2)$. Then there is a $\GF(q)$-representation of a matroid isomorphic to $\PG(3, q)$ of the form $(D_1 \; D_2)$ for some matrix $D_2$, and as $L_1 \cup L_2$ has rank three, there is a unique element indexing a column of $(D_1 \; D_2)$ that is in the closure of both $L_1$ and $L_2$. This column is parallel to $\alpha' - \omega\alpha$, contradicting the fact that it is not parallel to a vector over $\GF(q)$.
So there is at most one line $L$ of $M'' / f \d g$ such that $g \in \cl_{M''/f}(L)$, and there exists an element $e$ of $M'' / f \d g$ that is not in any such line, so $\cl_{M''/f}(\{e, g\}) = \{e, g\}$.
Therefore, $g$ is not in a parallel pair of $M'' / f, e$. Since $\si(M'' / f, e) \d g \cong \PG(2, q)$, it follows from (\ref{clm:mainpgtheorem-getlongline}) that $\si(M'' / f, e)$, and hence $M$, has a $U_{2, q^2+1}$-minor.
\end{proof}

\end{document}